\documentclass[11pt]{article}

\usepackage[a4paper,margin=1in]{geometry}
\usepackage[T1]{fontenc}
\usepackage[utf8]{inputenc}
\usepackage{lmodern}
\usepackage{amsmath,amssymb,amsthm,mathtools}
\usepackage{enumitem}
\usepackage[colorlinks=true,linkcolor=blue,citecolor=blue,urlcolor=blue]{hyperref}
\usepackage{cleveref}

\newcommand{\R}{\mathbb{R}}
\newcommand{\N}{\mathbb{N}}

\newcommand{\HH}{\mathbb{H}}

\newcommand{\E}{\mathbb{E}}

\newcommand{\F}{\mathcal{F}}

\newcommand{\eps}{\varepsilon}
\newcommand{\dd}{\,\mathrm{d}}
\newcommand{\D}{\mathbb{D}}

\newcommand{\graph}{\operatorname{graph}}

\DeclareMathOperator{\conv}{conv}

\DeclareMathOperator{\dist}{dist}

\newtheorem{theorem}{Theorem}
\newtheorem{proposition}{Proposition}
\newtheorem{lemma}{Lemma}
\newtheorem{corollary}{Corollary}

\theoremstyle{definition}
\newtheorem{definition}{Definition}
\newtheorem{assumption}{Assumption}

\theoremstyle{remark}
\newtheorem{remark}{Remark}

\crefname{theorem}{theorem}{theorems}
\Crefname{theorem}{Theorem}{Theorems}
\crefname{proposition}{proposition}{propositions}
\Crefname{proposition}{Proposition}{Propositions}
\crefname{lemma}{lemma}{lemmas}
\Crefname{lemma}{Lemma}{Lemmas}
\crefname{corollary}{corollary}{corollaries}
\Crefname{corollary}{Corollary}{Corollaries}
\crefname{definition}{definition}{definitions}
\Crefname{definition}{Definition}{Definitions}
\crefname{assumption}{assumption}{assumptions}
\Crefname{assumption}{Assumption}{Assumptions}
\crefname{example}{example}{examples}
\Crefname{example}{Example}{Examples}
\crefname{remark}{remark}{remarks}
\Crefname{remark}{Remark}{Remarks}

\title{Graphical stability of set-valued integrals under measure perturbations}
\author{Tam Le\ \thanks{Universit\'{e} Paris Cit\'e, LPSM}}
\date{July 20, 2026}

\begin{document}

\maketitle

\begin{abstract}
Motivated by sampling and approximation schemes arising in nonsmooth optimization, we study the stability of parameterized set-valued integrals under weak perturbations of the underlying probability distribution. For a compact parameter set, we show that compact convex-valued and jointly outer semicontinuous integrands induce set-valued integral maps that converge graphically in excess distance along any weakly convergent sequence of probability measures. The result holds under a superlinear integrability condition and provides a unified stability principle for measure approximations of set-valued expectations. We discuss the sharpness of the assumptions through examples. In particular, we emphasize that joint outer semicontinuity is required in general and that the superlinear envelope condition is tight relative to the classical i.i.d. empirical setting. As a consequence, we obtain outer stability of solution sets for stochastic generalized equations. We illustrate the stability result in several settings, including stochastic nonsmooth optimization with Markovian sampling, smoothing by mollifiers, and parameter-dependent distributional dynamics.
\end{abstract}



\section{Introduction}
Integrals and expectations of set-valued maps
\cite{Aumann1965,Debreu1967} are standard objects in nonsmooth stochastic
optimization and variational analysis
\cite{CastaingValadier1977,HiaiUmegaki1977,AubinFrankowska1990,Molchanov2005,RockafellarWets1998,benaim2005stochastic}.
They provide a natural language for averaging subdifferentials, random feasible
sets, and choices among possible responses.

In applications, the underlying probability law is often replaced by an
approximation or by a tractable surrogate. This leads to empirical
measures, occupation measures generated by Markovian sampling schemes
\cite{RamNedicVeeravalli2009,XieLuZhuWu2016,DuMordatch2019,Even2023}, and smoothing kernels used in nonsmooth
optimization and zeroth-order methods
\cite{ErmolievNorkinWets1995,Chen2012Smoothing,DuchiBartlettWainwright2012,NesterovSpokoiny2017}.
Related perturbations also appear in robust optimization \cite{ben2009robust}
and in models with decision-dependent data distributions
\cite{PerdomoZrnicMendlerDunnerHardt2020,EnnajiFadiliAttouch2024}.

The stability of such set-valued objects under perturbations of the underlying measure has therefore been studied in several settings. In particular, for i.i.d. empirical measures, graphical convergence and set-valued laws of large numbers are available in several forms: laws of large numbers for random sets and random
semicontinuous mappings \cite{ArtsteinVitale1975,NorkinWets2013}, uniform laws
for parameterized set-valued maps and subdifferentials \cite{ShapiroXu2007},
and strong laws for random monotone operators \cite{Salim2023}. The subdifferential case has been sharpened in a parallel line of work
\cite{DavisDrusvyatskiy2022,Ruan2025,TianRoyset2026} and related stability questions also arise for stochastic generalized equations, see e.g. \cite{LiuRomischXu2014}.

\bigskip

This paper studies the stability question under weak perturbations of the
integrating measure. Let \(S\) be a Polish space, let \(\Theta\subset\R^p\) be a
parameter space, and let
\[
G:\Theta\times S\rightrightarrows\R^d
\]
be a parameterized set-valued map. Given probability measures
\((\mu_k)_{k\in\N}\) and $\mu$ on \(S\) where \(\mu_k\) weakly converges to \(\mu\), we consider the set-valued integrals
\[
H_k(\theta):=\int_S G(\theta,s)\dd\mu_k(s)
\qquad \text{and} \qquad
H(\theta):=\int_S G(\theta,s)\dd\mu(s).
\]
We establish the convergence of the induced graphs
\[
\graph_\Theta H_k
:=
\{(\theta,x)\in\Theta\times\R^d:x\in H_k(\theta)\}
\xrightarrow[k \to \infty]{}
\graph_\Theta H .
\]
The convergence is measured by the \emph{excess distance} between the graphs, seen as compact subsets of \(\Theta\times\R^d\). This can be seen as a metric form of one-sided Painlev\'e--Kuratowski convergence
\cite[Chapter 4.B]{RockafellarWets1998} where every limit point of sequences
\((\theta_k,x_k)\in\graph_\Theta H_k\) belongs to \(\graph_\Theta H\). This
notion is natural in variational problems, where graphical convergence is used
to pass to limits in stationary sets \cite{NorkinWets2013} and in approximate
dynamics \cite{benaim2012perturbations}.

\bigskip

The result gives a unified treatment of several aspects that arise in such
approximations. It extends parameter-dependent graphical laws of large numbers
beyond the i.i.d. empirical setting \cite{ShapiroXu2007,NorkinWets2013} by allowing general weak perturbations of
the underlying probability law. It also allows unbounded integrands, a case
identified in \cite{artstein1988approximating} as a source of convergence
failure. In the present framework, this obstruction is controlled through a uniform integrability condition. The resulting stability principle covers
moving parameters, possibly unbounded values, and several approximation schemes relevant
to nonsmooth and stochastic optimization.

\paragraph{Organization of the paper} We recall several definitions and elementary results on uniform integrability and set-valued maps in \Cref{sec:prelims}. Our main result, \Cref{th:main_theorem}, is established in the subsequent \Cref{sec:main_result} along with consequences for ergodic sampling schemes, \Cref{cor:ergodic_sampling}, solutions of stochastic inclusions, \Cref{cor:stochastic_ge} and we provide a discussion on sharpness of our assumptions, see \Cref{rem:assumptions}. Finally, we illustrate our results with several applications in \Cref{sec:applications}.

\paragraph{Notation}
Throughout this work, \(S\) is a Polish space equipped with its Borel $\sigma$-algebra. Let
\(\mathbb{S}^{d-1}:=\{u\in\R^d:\|u\|=1\}\) and
\((r)_+:=\max\{r,0\}\). For nonempty compact sets
\(A,B\subset\R^d\), we define the excess distance $\D$ and the Hausdorff distance $\mathbb{H}$ by
\[
\D(A,B):=\sup_{a\in A}\dist(a,B),
\qquad
\HH(A,B):=\max \left\{\D(A,B),\D(B,A) \right\}.
\]
For \(\Theta\subset\R^p\), the graph of a set-valued map \(H\) restricted to
\(\Theta\) is denoted by
\[
\graph_\Theta H:=\{(\theta,x)\in\Theta\times\R^d:x\in H(\theta)\}.
\]
We write \(\mu_k\Rightarrow\mu\) to denote weak convergence of probability
measures, and \(\delta_x\) for the Dirac mass at \(x\).

\section{Preliminaries}
\label{sec:prelims}

\subsection{On uniform integrability}

We recall classical results on uniform integrability of a family of probability distributions.

\begin{lemma}[de la Vall\'{e}e-Poussin {\cite[Th.~22, p.~19]{dellacherie2011probabilities}}] \label{lem:de_la_vallee} Let $(\mu_k)_{k \in \N}$ be a sequence of probability measures. Let $X : S \to \R$ be measurable. Then the following are equivalent:
	\begin{enumerate}
		\item There exists a convex function $\Phi$ with $\Phi(t) /t \xrightarrow[t \to \infty]{} \infty$ such that \[\sup_{k \in \N} \int_S \Phi ( |X|) \dd \mu_k< \infty.\]
		\item \(X\) is uniformly integrable with respect to the sequence
		\((\mu_k)_{k\in\N}\), namely
		\[
		\lim_{M \to \infty}\sup_{k \in \N}
		\int_S |X| \mathbf{1}_{\{|X| \geq M\}} \dd \mu_k = 0.
		\]
	\end{enumerate}
\end{lemma}

A direct consequence when the sequence is reduced to a singleton is as follows.

\begin{lemma}[Existence of a superlinear integrable function] \label{lem:de_la_vallee_consequence}
	Let \(X:S\to\R\) be measurable and assume
	\(\int_S |X|\dd\mu <\infty\).  Then there exists a convex function
	\(\Phi:\R_+\to\R_+\) with \(\Phi(t)/t\xrightarrow[t \to \infty]{}\infty\) such that $\int_S \Phi (|X|) \dd\mu<\infty .$
\end{lemma}

\subsection{Set-valued analysis} In this part, we gather essential materials on set-valued maps.

\begin{definition}
	Let \((S,\F,\mu)\) be a probability space.  A map
	\(G:S\rightrightarrows \R^d\) admits a measurable selection if there exists an
	\(\F\)-measurable \(g:S\to\R^d\) such that \(g(s)\in G(s)\) for \(\mu\)-a.e.
	\(s\).  The Aumann integral of \(G\) with respect to \(\mu\) is
	\[
	\int_S G \dd\mu
	:=
	\left\{
	\int_S g \dd \mu:
	g \text{ is an integrable measurable selection of } G
	\right\}.
	\]
\end{definition}

\begin{definition}
	Let \(G:S\rightrightarrows\R^d\) have nonempty values.  We call
	\(G\) outer semicontinuous at \(s\in S\) if for every \(\eps>0\) there
	is a neighborhood \(U\) of \(s\) such that
	\[
	\D(G(s'),G(s))\leq \eps
	\qquad\forall s'\in U.
	\]
	It is outer semicontinuous if it is outer semicontinuous at every
	\(s\in S\).
\end{definition}
For $\Theta\subset\R^p$, and a set-valued map $H$, we recall that
\[
\graph_\Theta H:=\{(\theta,x)\in \Theta\times\R^q:x\in H(\theta)\}.
\]

\begin{lemma}[{adapted from \cite[Prop. 6.2]{aubin1987graphical}}]
	\label{lem:graph_convergence_pointwise_characterization}
	Let $\Theta\subset \mathbb{R}^p$ be nonempty and compact.  
	For each $k \in \N$, let
	$H_k:\mathbb{R}^p \rightrightarrows \mathbb{R}^q$ and
	$H:\mathbb{R}^p\rightrightarrows \mathbb{R}^q$ be outer semicontinuous,
	locally bounded, and nonempty compact valued on $\Theta$. Then the
	following are equivalent:
	\begin{enumerate}
		\item $\D(\graph_\Theta H_k,\ \graph_\Theta H) \xrightarrow[k \to \infty]{}0$;
		\item For any $\theta\in\Theta$ and any sequence
		$\theta_k\in\Theta$ with
		$\theta_k \xrightarrow[k \to \infty]{}\theta$,
		$\D\bigl(H_k(\theta_k),\,H(\theta)\bigr)
		\xrightarrow[k \to \infty]{} 0.$
	\end{enumerate}
\end{lemma}

\paragraph{Support functions}

We next recall how compact convex sets can be represented and compared through
their support functions. For a nonempty compact convex set \(K\subset\R^d\),
its support function \(\sigma_K:\R^d\to\R\) is defined by
\[
\sigma_K(u):=\sup_{x\in K}\langle u,x\rangle,
\qquad u\in\R^d.
\]

\begin{lemma}[{\cite[Lemma 17.30]{infinite}}] 
	Let $G : S \rightrightarrows \R^d$ be locally bounded and outer semicontinuous. Then for any $u \in \mathbb{S}^{d-1}$, $s \mapsto \sigma_{G(s)}(u)$ is upper semicontinuous.
\end{lemma}

\begin{lemma} \label{prop:support}
	Let \((S,\F,\mu)\) be a probability space, and let
	\(G:S\rightrightarrows\R^d\) be measurable, nonempty compact convex valued,
	and integrably bounded. Then
	\[
	\sigma_{\int_S G\dd\mu}(u)
	=
	\int_S \sigma_{G(s)}(u) \dd  \mu ( s) 
	\qquad \text{for all } u\in\R^d.
	\]
\end{lemma}

The following relationship between support functions and set distances is a key tool in laws of large numbers for set-valued maps \cite{ArtsteinVitale1975,ShapiroXu2007}.

\begin{lemma}[{\cite[Theorem II-18]{CastaingValadier1977}}]
	\label{lem:support_excess_hausdorff}
	Let \(A,B\subset\R^d\) be nonempty compact convex sets. Then
	\[
	\D(A,B)
	=
	\sup_{u\in\mathbb{S}^{d-1}}
	\big(\sigma_A(u)-\sigma_B(u)\big)_+
	\]
	and
	\[
	\HH(A,B)
	=
	\sup_{u\in\mathbb{S}^{d-1}}
	\left|\sigma_A(u)-\sigma_B(u)\right|.
	\]
\end{lemma}

\section{Main results}
\label{sec:main_result}

We now turn to the graphical stability of parameterized set-valued integrals.
Let \(\Theta\subset\R^p\) be nonempty and compact, and consider a set-valued map
\[
G:\Theta\times S\rightrightarrows\R^d.
\]
Let \((\mu_k)_{k\in\N}\) and \(\mu\) be probability measures on \(S\) such
that \(\mu_k\Rightarrow\mu\) as $k \to \infty$. We work under the following assumptions.

\begin{assumption}\label{ass:general_assumptions_D}
	  \begin{enumerate}
	  	\item[]
	  	\item{(Joint outer semicontinuity)} \(G:\Theta\times S\rightrightarrows\R^d\) is nonempty compact
	  	convex valued, jointly outer semicontinuous and locally bounded.
	  	\item{(Superlinear integrability)} There
	  	are a locally bounded function \(\kappa:\Theta\to\R_+\), a measurable
	  	function \(m:S\to\R_+\), and a convex function \(\Phi:\R_+\to\R_+\) such
	  	that
	  	\[
	  	\sup_{y\in G(\theta,s)}\|y\|\leq \kappa(\theta)m(s),
	  	\qquad
	  	\lim_{t\to\infty}\frac{\Phi(t)}{t}=\infty,
	  	\]
	  	and \(\sup_{k\in\N\cup\{\infty\}}
	  	\int_S \Phi(m(s))\dd\mu_k(s)<\infty\) where \(\mu_\infty:=\mu\).
	  \end{enumerate}
\end{assumption}
Compared with the i.i.d. empirical setting, where graphical convergence was
established in \cite{NorkinWets2013}, \Cref{ass:general_assumptions_D} requires outer semicontinuity jointly in the parameter and sample variables.  This condition is genuinely needed for general weakly convergent sequences
\((\mu_k)_{k \in \N}\); see \Cref{rem:assumptions}.  The superlinear integrability
condition in item 2, on the other hand, is essentially sharp with respect to
the i.i.d. case; see \Cref{cor:ergodic_sampling}. The graphical stability result states as follows.

\begin{theorem}[Graphical stability under weak measure convergence]
	\label{th:main_theorem}
	Let \(\Theta\subset\R^p\) be nonempty and compact. Under \Cref{ass:general_assumptions_D} where $\mu_k \Rightarrow \mu$ as $k \to \infty$, define for all $\theta \in \Theta$,
	\[
		H_k(\theta):=\int_S G(\theta,s)\dd\mu_k(s),
		\qquad
		H(\theta):=\int_S G(\theta,s)\dd\mu(s).
	\]
	Then
	\[
		\D\left(\graph_\Theta H_k,\graph_\Theta H\right)
		\xrightarrow[k\to\infty]{}0 .
	\]
\end{theorem}

\begin{remark}[On assumptions] \label{rem:assumptions} Our set of conditions in \Cref{ass:general_assumptions_D} is slightly more restrictive than those required when $\mu_k$ are standard empirical measures from i.i.d. samples \cite{NorkinWets2013}. Let us illustrate through examples that they are required in general. 
	\begin{enumerate}
		\item{On superlinear integrability.} Let \(S=\R\), \(\mu_k:=k^{-1}\delta_k+\left(1-k^{-1}\right)\delta_0\), and
		\(\mu:=\delta_0\).  Then \(\mu_k\Rightarrow\mu\).  Let
		\(G(s):=\{s\}\).  The integrals are singletons with 
		\[
		\int_S G(s)\dd\mu_k(s)=\{1\},
		\qquad
		\int_S G(s)\dd\mu(s)=\{0\}.
		\]
		Thus weak convergence of the measures alone does not imply convergence of
		the set-valued integrals.  In particular, we easily verify that the absolute value \(m(s)=|s|\) is integrable but not uniformly
		integrable along \((\mu_k)_{k \in \N}\), which is exactly what the superlinear
		condition excludes by \Cref{lem:de_la_vallee}.
		\item{On joint outer semicontinuity.} Let \(\Theta=[-1,1]\) and \(S=\R\). Let \(\mu_k:=\delta_{1/k}\) and
		\(\mu:=\delta_0\).  Then \(\mu_k\Rightarrow\mu\).  Define the
		compact convex-valued map
		\[
		G(\theta,s):=
		\begin{cases}
			[0,1], & \text{if } s=1/k\text{ for some }k\in\N
			\text{ and } |\theta|\leq 1/(2k),\\
			\{0\}, & \text{otherwise}.
		\end{cases}
		\]
		For every fixed \(s\), the section \(\theta\mapsto G(\theta,s)\) is clearly outer
		semicontinuous. Furthermore,  the
		map \(G\) is uniformly bounded by \(1\).
		However, \(G\) is not jointly outer semicontinuous at \((0,0)\), since
		\((0,1/k)\to(0,0)\), while \(1\in G(0,1/k)\) and
		\(1\notin G(0,0)\).  The corresponding integrals satisfy
		\[
		H_k(0)=\int_S G(0,s)\dd\mu_k(s)=[0,1],
		\qquad
		H(0)=\int_S G(0,s)\dd\mu(s)=\{0\}
		\]
		hence graphical convergence cannot hold. Thus partial outer semicontinuity in
		the parameter is not sufficient.
	\end{enumerate}
\end{remark}

We can specialize \Cref{th:main_theorem} to empirical measures generated by a possibly dependent sampling
process. More precisely, weak convergence of the empirical measures and the
required integrability can both be deduced from the following ergodicity assumption.

\begin{assumption}[Ergodicity]
	\label{ass:ergodic_sampling}
	 \(P\) is a probability measure on \(S\), and 
	\((\xi_t)_{t\geq1}\) is an \(S\)-valued stochastic process. For every
	$P$-integrable function \(f:S\to\R\), we have
	\[
	\frac1T\sum_{t=1}^T f(\xi_t)
	\xrightarrow[T \to \infty]{}
	\int_S f\dd P
	\qquad\text{almost surely}.
	\]
\end{assumption}
\Cref{ass:ergodic_sampling} covers i.i.d. sampling but it can also include  dependent
sampling schemes. In particular, it holds for positive Harris recurrent
Markov chains with invariant probability measure \(P\); see, e.g.,
\cite[Chapter~17]{MeynTweedie2009}.

\begin{corollary}[Graphical convergence for ergodic sampling]
	\label{cor:ergodic_sampling}
	Let \(\Theta\subset\R^p\) be nonempty and compact. Suppose that \(P\) and \((\xi_t)_{t\geq1}\) satisfy
	\Cref{ass:ergodic_sampling}. Let
	\(G:\Theta\times S\rightrightarrows\R^d\) be measurable, nonempty compact
	convex-valued, jointly outer semicontinuous, and locally bounded. Assume that
	there exist a locally bounded function \(\kappa:\Theta\to\R_+\) and a
	measurable function $m : S \to \R_+$, $P$-integrable, such that
	\[
	\sup_{y\in G(\theta,s)}\|y\|
	\leq \kappa(\theta)m(s).
	\]
	Define $H_T(\theta)
	:=
	\frac1T\sum_{t=1}^T G(\theta,\xi_t)$ and $
	H(\theta)
	:=
	\int_S G(\theta,s)\dd P(s).$ Then
	\[
	\D\bigl(\graph_\Theta H_T,\graph_\Theta H\bigr)
	\xrightarrow[T\to\infty]{}0
	\qquad\text{almost surely}.
	\]
\end{corollary}

\begin{proof}
	By \Cref{lem:de_la_vallee_consequence}, there exists a convex superlinear
	function \(\Phi:\R_+\to\R_+\) such that $\int_S\Phi(m(s))\dd P(s)<\infty.$ Applying \Cref{ass:ergodic_sampling} to \(\Phi\circ m\) gives
	\[
	\frac1T\sum_{t=1}^T\Phi(m(\xi_t))
	\longrightarrow
	\int_S\Phi(m(s))\dd P(s)
	\qquad\text{almost surely}.
	\]
	In particular, the sequence on the left-hand side is almost surely bounded. Moreover, since \(S\) is Polish, its topology admits a countable basis.
Let \(\mathcal U\) be the countable family consisting of all finite unions of
elements of this basis. Let $\widehat{P}_T := \frac{1}{T} \sum_{t=1}^T \delta_{{\xi_t}}$ Applying \Cref{ass:ergodic_sampling} to the indicator
of each \(U\in\mathcal U\), and using countability, gives $\widehat P_T(U)\longrightarrow P(U)$ $\text{for every }U\in\mathcal U$
simultaneously on an event of probability one. Every open subset of \(S\) is an increasing union of sets from
\(\mathcal U\). Consequently \cite[Theorem 2.2]{Billingsley1999} gives the weak convergence \(\widehat P_T\Rightarrow P\) almost surely. Therefore, for almost every sample path, the measures
	\(\mu_T=\widehat P_T\) and \(\mu=P\) satisfy
	\Cref{ass:general_assumptions_D}. The conclusion follows from
	\Cref{th:main_theorem}.
\end{proof}

Now, let us state a consequence of \Cref{th:main_theorem} for the stability of generalized stochastic equations.

\begin{corollary}[Stability of stochastic generalized equations]\label{cor:stochastic_ge}
Let \(\Theta\subset\R^p\) be nonempty and compact. Under \Cref{ass:general_assumptions_D} with $\mu_k \Rightarrow \mu$ as $k \to \infty$, let
$N:\Theta\rightrightarrows\R^d$ be closed-valued and outer semicontinuous.
Consider
\begin{equation*}
Z_k = \left\{	\theta \in \Theta \ : \ 0 \in \int_S G(\theta,s) \dd \mu_k(s) + N(\theta)\right\} \qquad \text{for all } k \in \N,
\end{equation*}
\begin{equation*}
	Z = \left\{	\theta \in \Theta \ : \ 0 \in \int_S G(\theta,s) \dd \mu(s) + N(\theta)\right\}.
\end{equation*}
If for all $k \in \N$, $Z_k$ and \(Z\) are nonempty, then $\D(Z_k,Z) \xrightarrow[k \to \infty]{} 0.$
\end{corollary}

This contrasts with previous stability results for stochastic generalized equations such
as \cite{LiuRomischXu2014}, where perturbations are measured through a
problem-dependent discrepancy. In our case, the starting point is weak convergence of probability measures, the natural
mode of convergence for the underlying laws, and the assumptions identify when
it is strong enough to imply stability.

\begin{proof}
	Write
	\[
		H_k(\theta):=\int_S G(\theta,s)\dd\mu_k(s),
		\qquad
		H(\theta):=\int_S G(\theta,s)\dd\mu(s),
	\]
	and
	\[
		I_k(\theta):=H_k(\theta)+N(\theta),
		\qquad
		I(\theta):=H(\theta)+N(\theta).
	\]
	We first verify the graphical hypothesis required in
	\cite[Theorem 5.37(a)]{RockafellarWets1998}.  Namely, take any sequence
	\(k_j\to\infty\) and any convergent graph points
	\[
		\theta_j\to\theta,\qquad z_j\to z,\qquad
		z_j\in I_{k_j}(\theta_j).
	\]
	We show that \(z\in I(\theta)\). Under \Cref{ass:general_assumptions_D},
	\[
		M:=\sup_{k\in\N\cup\{\infty\}}\int_S m(s)\dd\mu_k(s)<\infty,
	\]
	where \(\mu_\infty:=\mu\).  Since \(\Theta\) is compact and \(\kappa\) is
	locally bounded, we have, with \(H_\infty:=H\),
	\[
		\sup_{k\in\N\cup\{\infty\}}\sup_{\theta\in\Theta}
		\sup_{h\in H_k(\theta)}\|h\|
		\leq
		M\sup_{\theta\in\Theta}\kappa(\theta)<\infty .
	\]
	Choose
	\[
		h_j\in H_{k_j}(\theta_j),
		\qquad
		n_j\in N(\theta_j),
		\qquad
		z_j=h_j+n_j .
	\]
	The sequence \((h_j)_{j \in \N}\) is bounded by the estimate above.  Since \((z_j)_{j \in \N}\)
	converges, \((n_j)_{j \in \N}\) is bounded as well.  Passing to a subsequence, assume
	that \(h_j\to h\) and \(n_j\to n\). By \Cref{th:main_theorem}, \(h\in
	H(\theta)\). By outer semicontinuity of \(N\), \(n\in N(\theta)\).  Hence
	\(z=h+n\in I(\theta)\), as required.  Moreover \(I_k\) and \(I\) are
	closed-valued, since they are sums of compact and closed subsets of a
	finite-dimensional space.
	
	We may therefore apply \cite[Theorem 5.37(a)]{RockafellarWets1998} to the
	generalized equations \(I_k(\theta)\ni0\) and \(I(\theta)\ni0\).  It follows
	that every cluster point of a sequence \(\theta_j\in Z_{k_j}\),
	\(k_j\to\infty\), belongs to \(Z\).
	
	
	Assume now that \(Z\neq\emptyset\).  If \(\D(Z_k,Z)\not\to0\), then for some
	\(\varepsilon>0\) there are \(k_j\to\infty\) and \(\theta_j\in Z_{k_j}\) such
	that \(\dist(\theta_j,Z)\geq\varepsilon\).  By compactness, after passing to a
	subsequence, \(\theta_j\to\theta\).  The conclusion of
	\cite[Theorem 5.37(a)]{RockafellarWets1998} gives \(\theta\in Z\), and hence
	\(\dist(\theta_j,Z)\leq\|\theta_j-\theta\|\to0\), which is a contradiction.  Therefore
	\(\D(Z_k,Z)\to0\).
\end{proof}

\paragraph{Proof of the main result} The proof of \Cref{th:main_theorem} proceeds by establishing the pointwise
criterion in \Cref{lem:graph_convergence_pointwise_characterization} after
lifting the parameter into the probability space as a moving Dirac mass.
We therefore begin with a pointwise convergence lemma, which extends the
approximation result of Artstein and Wets \cite[Theorem 4.2]{artstein1988approximating} to
outer semicontinuous and possibly unbounded integrable maps.

\begin{lemma}[Pointwise outer convergence] \label{lem:pointwise} Let $\widetilde{G} : S \rightrightarrows \R^d$ be nonempty
	compact convex valued, outer semicontinuous, and locally bounded. Let
	$\gamma_k \Rightarrow \gamma$ as $k \to \infty$. Assume there exists a
	measurable function $\tilde{m} : S \to \R_+$ such that $\sup_{x\in \widetilde{G}(s)}\|x\|\leq \tilde{m}(s)$ and a convex function $\Phi  : \R_+ \to \R_+$ such that
	$\lim_{t \to \infty} \frac{\Phi(t)}{t} = \infty$ and, with
	$\gamma_\infty:=\gamma$, $\sup_{k \in \N\cup\{\infty\}}
	\int_S \Phi \circ \tilde{m} \dd \gamma_k < \infty.$ Then $\D \left( \int_S \widetilde{G} \dd \gamma_k, \int_S \widetilde{G} \dd \gamma \right) \xrightarrow[k \to \infty]{} 0.$
\end{lemma}

\begin{proof}  Note that $\widetilde{G}$ is Borel measurable because outer semicontinuous, nonempty compact valued, and $\R^d$ is $\sigma$-compact; see, e.g., \cite[Theorem 18.10]{infinite}. Set
	\[
	\widetilde{H}_k:=\int_S \widetilde{G}\dd\gamma_k,
	\qquad
	\widetilde{H}:=\int_S \widetilde{G}\dd\gamma .
	\]
	The support function properties, see \Cref{prop:support} and \Cref{lem:support_excess_hausdorff}, give us 
	
	\begin{equation*}
		\D \left( \widetilde{H}_k, \widetilde{H} \right)
		= \sup_{\|u\| = 1} \left(\sigma_{\widetilde{H}_k}(u)-\sigma_{\widetilde{H}}(u)\right)_{+}
		= \sup_{\|u\| = 1} \left(\int_S \sigma_{\widetilde{G}(s)}(u) \dd \gamma_k(s) - \int_S \sigma_{\widetilde{G}(s)}(u) \dd \gamma(s)\right)_{+}.
	\end{equation*}
	By outer semicontinuity and local boundedness of $\widetilde{G}$, $s \mapsto \sigma_{\widetilde{G}(s)}(u)$ is upper semicontinuous. Let us show that for any $u \in \mathbb{S}^{d-1}$,
	\begin{equation}
		\label{eq:fixed_u_limsup_sigma}
		\limsup_{k \to \infty} \int_S \sigma_{ \widetilde{G}(s)} (u) \dd \gamma_k(s)  \leq \int_S \sigma_{\widetilde{G}(s)}(u) \dd \gamma(s).  
	\end{equation}
	Indeed, for $M > 0$, consider the upper semicontinuous truncations
	\[
	\psi_M(s):=\max\{-M,\min\{\sigma_{\widetilde{G}(s)}(u),M\}\}.
	\]
	Then we can write 
	\begin{equation*}
		\sigma_{\widetilde{G}(s)}(u) = \sigma_{\widetilde{G}(s)}(u) \mathbf{1}_{\{|\sigma_{\widetilde{G}(s)}(u)|> M\}} + \psi_M(s) \mathbf{1}_{\{|\sigma_{\widetilde{G}(s)}(u)|\leq M\}}. 
	\end{equation*}
	We compare $\sigma_{\widetilde{G}(\cdot)}(u)$ with $\psi_M$ after integration.  Since
	$\psi_M(s)=\sigma_{\widetilde{G}(s)}(u)$ whenever
	$|\sigma_{\widetilde{G}(s)}(u)|\leq M$,
	\[
	0\leq
	\left|\sigma_{\widetilde{G}(s)}(u)-\psi_M(s)\right|
	\leq
	|\sigma_{\widetilde{G}(s)}(u)|
	\mathbf{1}_{\{|\sigma_{\widetilde{G}(s)}(u)|>M\}} .
	\]
	Moreover, $|\sigma_{\widetilde{G}(s)}(u)|\leq \tilde{m}(s)$, and therefore
	$\{|\sigma_{\widetilde{G}(s)}(u)|>M\}\subset \{\tilde{m}(s)>M\}$.  Thus, for every
	$k\in\N$,
	\[
	\int_S \sigma_{\widetilde{G}(s)}(u)\dd\gamma_k(s)
	\leq
	\int_S \psi_M(s)\dd\gamma_k(s)
	+
	\int_S \tilde{m}(s)\mathbf{1}_{\{\tilde{m}(s)>M\}}\dd\gamma_k(s).
	\]
	The de la Vallee-Poussin criterion, see
	\Cref{lem:de_la_vallee}, gives the uniform tail estimate
	\[
	\lim_{M\to\infty}
	\sup_{k\in\N\cup\{\infty\}}
	\int_S \tilde{m}(s)\mathbf{1}_{\{\tilde{m}(s)>M\}}\dd\gamma_k(s)
	=0.
	\]
	On the other hand, $\psi_M$ is bounded and upper semicontinuous, so the
	Portmanteau theorem gives
	\[
	\limsup_{k\to\infty}
	\int_S \psi_M(s)\dd\gamma_k(s)
	\leq
	\int_S \psi_M(s)\dd\gamma(s).
	\]
	Combining the two estimates, for every fixed $M>0$,
	\[
	\limsup_{k\to\infty}
	\int_S \sigma_{\widetilde{G}(s)}(u)\dd\gamma_k(s)
	\leq
	\int_S \psi_M(s)\dd\gamma(s)
	+
	\sup_{j\in\N\cup\{\infty\}}
	\int_S \tilde{m}(s)\mathbf{1}_{\{\tilde{m}(s)>M\}}\dd\gamma_j(s).
	\]
	Finally, $|\psi_M|\leq \tilde{m}$ and $\psi_M(s)\to\sigma_{\widetilde{G}(s)}(u)$ pointwise.
	Since $\tilde{m}$ is $\gamma$-integrable, dominated convergence yields
	\[
	\int_S \psi_M(s)\dd\gamma(s)
	\xrightarrow[M\to\infty]{}
	\int_S \sigma_{\widetilde{G}(s)}(u)\dd\gamma(s).
	\]
	Letting $M\to\infty$ gives \eqref{eq:fixed_u_limsup_sigma}. Finally, \eqref{eq:fixed_u_limsup_sigma} writes
	\[
	\limsup_{k \to \infty}
	\left(\sigma_{\widetilde{H}_k}(u)-\sigma_{\widetilde{H}}(u)\right)_{+}=0
	\qquad
	\text{for every }u\in\mathbb{S}^{d-1}.
	\]
	
	It remains to pass from pointwise convergence on the sphere to uniform
	convergence.  The same envelope gives a uniform Lipschitz bound.  Since
	$\Phi(t)/t\to\infty$, there is a constant
	\[
	C:=\sup_{k\in\N\cup\{\infty\}}\int_S \tilde{m}(s)\dd\gamma_k(s)<\infty .
	\]
	Therefore, for all $u,u'\in\mathbb{S}^{d-1}$,
	\[
	|\sigma_{\widetilde{H}_k}(u')-\sigma_{\widetilde{H}_k}(u)|
	\leq
	\int_S |\sigma_{\widetilde{G}(s)}(u')-\sigma_{\widetilde{G}(s)}(u)|\dd\gamma_k(s)
	\leq
	C\|u'-u\|.
	\]
	The same estimate holds for $\sigma_{\widetilde{H}}$.  Let $\eps>0$ and choose a finite
	$\delta$-net $\{u_1,\ldots,u_N\}$ of $\mathbb{S}^{d-1}$ with
	$2C\delta<\eps$.  For any $u\in\mathbb{S}^{d-1}$, choose $u_i$ with
	$\|u-u_i\|\leq\delta$.  Then
	\[
	\left(\sigma_{\widetilde{H}_k}(u)-\sigma_{\widetilde{H}}(u)\right)_{+}
	\leq
	\left(\sigma_{\widetilde{H}_k}(u_i)-\sigma_{\widetilde{H}}(u_i)\right)_{+}
	+2C\delta .
	\]
	Taking the supremum over $u$ and then the limsup in $k$ yields
	\[
	\limsup_{k\to\infty}\D(\widetilde{H}_k,\widetilde{H})
	\leq
	\max_{1\leq i\leq N}
	\limsup_{k\to\infty}
	\left(\sigma_{\widetilde{H}_k}(u_i)-\sigma_{\widetilde{H}}(u_i)\right)_{+}
	+2C\delta
	\leq \eps.
	\]
	Since $\eps>0$ was arbitrary, $\D(\widetilde{H}_k,\widetilde{H})\to0$.	
\end{proof}

\bigskip

We may now deduce \Cref{th:main_theorem}.

\begin{proof}[Proof of \Cref{th:main_theorem}]
	We verify the sequential condition in
	\Cref{lem:graph_convergence_pointwise_characterization}.  Let
	$\theta_k\to\theta$ in $\Theta$.  Define
	\[
	\widetilde G:\Theta\times S\rightrightarrows\R^d,
	\qquad
	\widetilde G(\vartheta,s):=G(\vartheta,s),
	\]
	and set
	\[
	\gamma_k:=\delta_{\theta_k}\otimes \mu_k,
	\qquad
	\gamma:=\delta_{\theta}\otimes \mu .
	\]
	By Slutsky's theorem, $\gamma_k\Rightarrow \gamma$ on
	$\Theta\times S$.  Since $\Theta$ is compact and $\kappa$ is locally
	bounded, there exists $K\geq1$ such that
	$\kappa(\vartheta)\leq K$ for all $\vartheta\in\Theta$.  Hence
	\[
	\sup_{x\in \widetilde G(\vartheta,s)}\|x\|
	\leq K m(s).
	\]
	With $\widetilde m(\vartheta,s):=Km(s)$ and
	$\widetilde\Phi(t):=\Phi(t/K)$, we have
	\[
	\sup_{k\in\N\cup\{\infty\}}
	\int_{\Theta\times S}\widetilde\Phi\circ \widetilde m
	\dd\gamma_k
	=
	\sup_{k\in\N\cup\{\infty\}}
	\int_S \Phi\circ m \dd\mu_k
	<\infty .
	\]
	The pointwise convergence from \Cref{lem:pointwise} applies to $\widetilde G$ and gives
	\[
	\D\left(
	\int_{\Theta\times S}\widetilde G\dd\gamma_k,
	\int_{\Theta\times S}\widetilde G\dd\gamma
	\right)\to0 .
	\]
	Note that by Fubini theorem for set-valued maps, see \cite[Theorem 2.1]{ZHANG1994355fubini} or \cite[Proposition 2.6]{LiuRomischXu2014}, we can write
	\[
	\int_{\Theta\times S}\widetilde G\dd(\delta_{\theta_k}\otimes\mu_k)
	=
	\int_S G(\theta_k,s)\dd\mu_k(s)
	=
	H_k(\theta_k),
	\]
	and similarly
	\[
	\int_{\Theta\times S}\widetilde G\dd(\delta_{\theta}\otimes\mu)
	=
	H(\theta).
	\]
	 Hence for every sequence $\theta_k\to\theta$ in $\Theta$,
	\[
	\D\bigl(H_k(\theta_k),H(\theta)\bigr)\to0.
	\]
	The same argument with the measure sequence fixed shows that each $H_k$ and
	$H$ is outer semicontinuous on $\Theta$.  Standard properties of the
	Aumann integral give nonempty compact convex values, and the integrable bound $m$ gives local boundedness: indeed,
	\[
	\sup_{y\in H_k(\vartheta)}\|y\|
	\leq
	\kappa(\vartheta)\int_S m(s)\dd\mu_k(s),
	\qquad
	\sup_{y\in H(\vartheta)}\|y\|
	\leq
	\kappa(\vartheta)\int_S m(s)\dd\mu(s).
	\]
	The graphical characterization from	\Cref{lem:graph_convergence_pointwise_characterization} now yields the desired result,
	\[
	\D(\graph_\Theta H_k,\graph_\Theta H)\to0 .
	\]
\end{proof}

\paragraph{On two-sided pointwise convergence} The one-sided excess in \Cref{lem:pointwise} is the natural consequence of
outer semicontinuity and Portmanteau theorem. Now,  if the support
functions are continuous at the points charged by the limiting measure, then
the same argument gives Hausdorff convergence, as in the fully continuous setting \cite[Theorem 3.1]{artstein1988approximating}.

The semialgebraic setting provides a directly verifiable sufficient condition.
Indeed, when \(S=\R^q\) and \(G\) is semialgebraic and outer semicontinuous,
its points of discontinuity form a set of dimension at most \(q-1\)
\cite[Theorem~28]{daniilidis2011continuity}. This set is negligible for any
probability measure absolutely continuous with respect to Lebesgue measure.
The following proposition records both the general criterion and this
ready-to-use case.

\begin{proposition}[Pointwise Hausdorff convergence]\label{prop:pointwise_hausdorff}
In the setting of \Cref{lem:pointwise}, assume furthermore that $\widetilde{G}$ is continuous with respect to the Hausdorff distance $\gamma$-almost everywhere.  Then
	\[
	\HH\left(\int_S \widetilde{G}\dd\gamma_k,\int_S \widetilde{G}\dd\gamma\right)
	\xrightarrow[k\to\infty]{}0 .
	\]
In particular, this holds if $\widetilde{G}$ is semialgebraic and $\gamma$ is absolutely continuous with respect to Lebesgue measure.
\end{proposition}

\begin{proof}
	Fix \(u\in\mathbb{S}^{d-1}\). By continuity of $\widetilde{G}$ $\gamma$-almost everywhere, the support function
	\(s\mapsto\sigma_{\widetilde{G}(s)}(u)\) is
	\(\gamma\)-almost everywhere continuous. This holds for instance by Berge maximum theorem \cite[Theorem 17.31]{infinite}.  It is also dominated by the integrable function \(m\) given under the setting of \Cref{lem:pointwise}. Hence the Portmanteau theorem, combined with the same truncation and uniform-integrability
	argument used in the proof of \Cref{lem:pointwise}, gives
	\[
	\int_S \sigma_{\widetilde{G}(s)}(u)\dd\gamma_k(s)
	\xrightarrow[k \to \infty]{}
	\int_S \sigma_{\widetilde{G}(s)}(u)\dd\gamma(s).
	\]
	The support-function representation therefore gives pointwise convergence
	of \(\sigma_{\widetilde{H}_k}\) to \(\sigma_{\widetilde{H}}\) on \(\mathbb{S}^{d-1}\), where $\widetilde{H}_k = \int_S \widetilde{G} \dd \gamma_k$, $\widetilde{H} = \int_S \widetilde{G} \dd \gamma$.  The integrable bound 
	\(\tilde{m}\) gives the uniform Lipschitz estimate
	\[
	|\sigma_{\widetilde{H}_k}(u)-\sigma_{\widetilde{H}_k}(v)|
	\leq
	C\|u-v\|,
	\qquad
	C:=\sup_{k\in\N\cup\{\infty\}}\int_S \tilde{m}\dd\gamma_k<\infty,
	\]
	and the same estimate holds for \(\sigma_{\widetilde{H}}\).  A finite covering argument on
	\(\mathbb{S}^{d-1}\), as in the proof of \Cref{lem:pointwise}, implies the uniform convergence $\sup_{\|u\|=1}
	|\sigma_{\widetilde{H}_k}(u)-\sigma_{\widetilde{H}}(u)|
	\to0 .$
	For nonempty compact convex subsets of \(\R^d\), this last quantity is
	exactly the Hausdorff distance (\Cref{lem:support_excess_hausdorff}). 
	
	The last part of the statement is a consequence of \cite[Theorem~28]{daniilidis2011continuity}.
\end{proof}

\section{Applications to nonsmooth and stochastic optimization}

\label{sec:applications}

This part is dedicated to applications of the general results from \Cref{sec:main_result}. We will use the notion of \emph{Clarke subdifferential}, which is defined for any locally Lipschitz function \(F:\R^p\to\R\) as the set-valued map
\[
\partial^c F(x)
:=
\conv\left\{
v\in\R^p:
\exists x_k\to x,\ x_k\in\operatorname{diff}F,\
\nabla F(x_k)\to v
\right\},
\]
where \(\operatorname{diff}F\) is the differentiability set of \(F\), which
has full Lebesgue measure by Rademacher's theorem.

\subsection{Sample stability of nonsmooth stochastic problems}
We illustrate \Cref{cor:stochastic_ge} through a nonsmooth stochastic problem. Take $\Theta \subset \R^p$ compact and convex, \(S=\R^q\) and consider the minimization problem
\[
\min_{\theta\in\Theta}\; \E_{\xi\sim P}[f(\theta,\xi)],
\]
where \(f:\R^p\times\R^q\to\R\) is jointly locally Lipschitz and bounded below.  For a fixed sample \(s\), a nonsmooth first-order method typically uses a
generalized gradient of the sampled loss \(\theta\mapsto f(\theta,s)\). Here, this suggests using the partial Clarke subdifferential $\partial^c_\theta f(\theta,s)
:=
\partial^c(f(\cdot,s))(\theta).$ However,
this object need not be outer semicontinuous jointly in \((\theta,s)\).  Instead, we may consider the projected Clarke subdifferential. Let $\partial^c f: \R^{p} \times  \R^q \rightrightarrows \R^{p} \times  \R^q$
denote the joint Clarke subdifferential of \(f\), and for
\((\theta,s)\in\R^p\times\R^q\) define the projected map
\[
D_f(\theta,s)
:=
\left\{
v_1\in\R^p:
\exists v_2\in\R^q\text{ such that }(v_1,v_2)\in\partial^c f(\theta,s)
\right\}.
\]
When \(f\) is, for instance, semialgebraic, a chain rule holds
\cite{bolte2021conservative}  for each fixed \(s\): for any absolutely continuous curve
\(x:[0,1]\to\R^p\),  for Lebesgue-almost all $t \in [0,1]$,
\[
\frac{\dd f(x(t),s)}{\dd t}
=
\langle v,\dot{x}(t)\rangle,
\qquad
\text{for all }v\in D_f(x(t),s).
\]
Semismoothness properties also hold \cite{bolte2009tame,davis2022conservative};
for each fixed \(s\), at each $\theta$,
\[
f(z,s)
=
f(\theta,s)+\langle v,z-\theta\rangle+o(\|\theta-z\|),
\qquad
v\in D_f(z,s).
\]
These generalized first-order regularity properties underlie convergence
analyses for nonsmooth nonconvex stochastic subgradient methods toward
generalized critical points
\cite{ermol1998stochastic,davis2020stochastic,bolte2023subgradient}. In the
present setting, these critical points are defined as
\[
Z
:=
\left\{
\vartheta\in\Theta:
0\in\E_{\xi\sim P}[D_f(\vartheta,\xi)]+N_\Theta(\vartheta)
\right\},
\]
where \(N_\Theta\) is the normal cone to the convex domain \(\Theta\). When an approximate distribution \(\widehat P_T\) is used in place of $P$, our convergence
results guarantee stability of the approximate critical points as
\(\widehat P_T\Rightarrow P\).  More precisely, if for each \(T\geq1\) we define
\[
Z_T
:=
\left\{
\vartheta\in\Theta:
0\in\E_{\xi\sim\widehat P_T}[D_f(\vartheta,\xi)]
+N_\Theta(\vartheta)
\right\},
\]
then $\D(Z_T,Z)\xrightarrow[T\to\infty]{}0$
under the conditions of \Cref{cor:stochastic_ge} or the ergodic setting of \Cref{cor:ergodic_sampling}.

\subsection{Smoothing by mollifiers}

A standard method to smooth a nonsmooth function is to average its values under
small perturbations
\cite{ErmolievNorkinWets1995,Chen2012Smoothing,NesterovSpokoiny2017}. A natural question is whether the resulting gradients remain consistent with
the subdifferential of the original function. We show that this follows  from the graphical stability established in \Cref{th:main_theorem}.

Let \(F:\R^p\to\R\) be locally Lipschitz. For \(\alpha>0\), let
\(\rho_\alpha\) be a probability density on \(\R^p\), and denote
\[
\mu_\alpha(\dd s):=\rho_\alpha(s)\dd s.
\]
Assume that \(\mu_\alpha\Rightarrow\delta_0\) as \(\alpha\downarrow0\).
Examples include uniform distributions on balls of radius \(\alpha\) and
centered Gaussian distributions with covariance \(\alpha^2I_p\). The
corresponding smoothed function is
\[
F_\alpha(\theta)
:=
\int_{\R^p}F(\theta-s)\rho_\alpha(s)\dd s.
\]
and the natural first-order object for \(F_\alpha\) is the averaged Clarke subdifferential
\[
H_\alpha(\theta)
:=
\int_{\R^p}\partial^cF(\theta-s)\rho_\alpha(s)\dd s.
\]
To apply \Cref{th:main_theorem}, consider
\[
G(\theta,s):=\partial^cF(\theta-s).
\]
The Clarke subdifferential of a locally Lipschitz function has nonempty compact
convex values and is outer semicontinuous and locally bounded. Consequently,
this choice of \(G\) is measurable, jointly outer semicontinuous, and locally bounded on
\(\Theta\times\R^p\).

It remains to verify the superlinear integrability condition. For compactly
supported kernels, suppose that the supports of \(\mu_\alpha\), for
\(0<\alpha\leq\alpha_0\), are contained in a fixed compact set \(K\). Taking
\(K\) as the sample space, the points \(\theta-s\) range over the compact set
\(\Theta-K\), on which \(\partial^cF\) is uniformly bounded. The condition
then holds with a constant bound $m$.

For kernels with noncompact support, a growth condition is needed. For
example, assume that
\begin{equation}
	\label{eq:poly_growth_F}
	\sup_{v\in\partial^cF(x)}\|v\|
\leq C(1+\|x\|^r)
\end{equation}
for some \(r\geq0\). Since \(\Theta\) is compact, one may take an envelope of
the form \(m(s)=C'(1+\|s\|^r)\). For Gaussian kernels, the required
superlinear integrability follows from their uniformly bounded moments, for
instance by taking \(\Phi(t)=t^{1+\varepsilon}\) with
\(\varepsilon>0\). Such growth conditions are readily checked, for instance \eqref{eq:poly_growth_F} holds for semialgebraic $F$.

Under either of these conditions, and more generally whenever
\Cref{ass:general_assumptions_D} holds, applying
\Cref{th:main_theorem} along any sequence \(\alpha_j\downarrow0\) gives $\D\left(
\graph_\Theta H_\alpha,
\graph_\Theta\partial^cF
\right)
\xrightarrow[\alpha\downarrow0]{}0.$
This graphical convergence yields the usual gradient-consistency statement.
When differentiation under the integral is justified, \(F_\alpha\) is
differentiable and
\[
\nabla F_\alpha(\theta)
=
\int_{\R^p}\nabla F(\theta-s)\rho_\alpha(s)\dd s
\in H_\alpha(\theta),
\]
because
\(\nabla F(\theta-s)\in\partial^cF(\theta-s)\) for Lebesgue-almost every
\(s\). Hence gradient consistency holds:
\[
\sup_{\theta\in\Theta}
\dist\left(
(\theta,\nabla F_\alpha(\theta)),
\graph_\Theta\partial^cF
\right)
\xrightarrow[\alpha\downarrow0]{}0.
\]
\subsection{Solutions of differential inclusions with parameter-dependent laws}

The continuous-time behavior of nonsmooth stochastic algorithms is naturally
described by differential inclusions driven by set-valued mean fields
\cite{benaim2005stochastic,benaim2012perturbations,davis2020stochastic}.
In decision-dependent and performative models, the distribution of the data
may itself depend on the current parameter
\cite{PerdomoZrnicMendlerDunnerHardt2020,DrusvyatskiyXiao2023,
	EnnajiFadiliAttouch2024}. Stability under moving probability laws is therefore
directly relevant: it provides the outer semicontinuity needed to apply the
standard existence theory for such dynamics.

\bigskip

Let \((\mu_\theta)_{\theta\in\R^p}\) be a parameterized family of probability
measures on \(S\), and consider the set-valued mean field
\[
V(\theta)
:=
\int_S G(\theta,s)\dd\mu_\theta(s),
\qquad
G:\R^p\times S\rightrightarrows\R^p.
\]
Assume that \(\theta\mapsto\mu_\theta\) is weakly continuous:
\[
\theta_k\to\theta
\quad\Longrightarrow\quad
\mu_{\theta_k}\Rightarrow\mu_\theta.
\]
Suppose that \(G\) satisfies the regularity conditions of
\Cref{ass:general_assumptions_D} and that its integrability condition holds
locally uniformly in \(\theta\). More precisely, for every compact
\(K\subset\R^p\), there exist \(C_K>0\), a measurable function
\(m_K:S\to\R_+\), and a convex function \(\Phi_K:\R_+\to\R_+\) such that
\[
\sup_{y\in G(\vartheta,s)}\|y\|
\leq C_Km_K(s)
\qquad
\text{for all }(\vartheta,s)\in K\times S,
\]
\[
\lim_{t\to\infty}\frac{\Phi_K(t)}{t}=\infty,
\qquad
\sup_{\vartheta\in K}
\int_S\Phi_K(m_K(s))\dd\mu_\vartheta(s)<\infty.
\]
Let \(\theta_k\to\theta\), and choose a compact set \(K\) containing
\(\theta\) and the sequence \((\theta_k)_{k \in \N}\). For \(\vartheta\in K\), define
\[
H_k(\vartheta)
:=
\int_S G(\vartheta,s)\dd\mu_{\theta_k}(s),
\qquad
H(\vartheta)
:=
\int_S G(\vartheta,s)\dd\mu_\theta(s).
\]
The hypotheses of \Cref{th:main_theorem} hold on \(K\), and hence
\[
\D(\graph_K H_k,\graph_K H)\longrightarrow0.
\]
Using the pointwise characterization from \Cref{lem:graph_convergence_pointwise_characterization} with
\(\vartheta_k=\theta_k\), we obtain
\[
\D\bigl(V(\theta_k),V(\theta)\bigr)
=
\D\bigl(H_k(\theta_k),H(\theta)\bigr)
\longrightarrow0.
\]
Thus \(V\) is outer semicontinuous. Consequently, if \(V\) has nonempty compact convex values and satisfies a
linear-growth bound, the differential inclusion
\begin{equation*}
	\dot\theta(t)\in V(\theta(t))
	\quad\text{for almost every }t\geq0,
	\qquad
	\theta(0)=\theta_0,
\end{equation*}
falls within the standard existence theory, see \cite{aubin2008differential}. In particular, it admits a global 
absolutely continuous solution.

\section*{Acknowledgements}

The author used ChatGPT (OpenAI) as an aid in reviewing mathematical arguments and brainstorming possible counterexamples. All proofs, claims, and examples were independently checked and verified by the author, who assumes full responsibility for the manuscript.

\bibliographystyle{abbrv}
\bibliography{references}

\begin{thebibliography}{10}

\bibitem{infinite}
C.~D. Aliprantis and K.~C. Border.
\newblock {\em Infinite Dimensional Analysis}.
\newblock Springer, Berlin, 2006.
\newblock \url{https://doi.org/10.1007/3-540-29587-9}.

\bibitem{ArtsteinVitale1975}
Z.~Artstein and R.~A. Vitale.
\newblock A strong law of large numbers for random compact sets.
\newblock {\em The Annals of Probability}, 3(5):879--882, 1975.
\newblock \url{https://doi.org/10.1214/aop/1176996275}.

\bibitem{artstein1988approximating}
Z.~Artstein and R.~J.-B. Wets.
\newblock Approximating the integral of a multifunction.
\newblock {\em Journal of Multivariate Analysis}, 24(2):285--308, 1988.
\newblock \url{https://doi.org/10.1016/0047-259X(88)90041-3}.

\bibitem{aubin1987graphical}
J.-P. Aubin.
\newblock Graphical convergence of set-valued maps.
\newblock Technical Report WP-87-083, International Institute for Applied
  Systems Analysis, Laxenburg, Austria, 1987.

\bibitem{AubinFrankowska1990}
J.-P. Aubin and H.~Frankowska.
\newblock {\em Set-Valued Analysis}.
\newblock Systems \& Control: Foundations \& Applications. Birkh{\"a}user,
  Boston, 1990.
\newblock \url{https://doi.org/10.1007/978-0-8176-4848-0}.

\bibitem{aubin2008differential}
J.-P. Aubin and H.~Frankowska.
\newblock Differential inclusions.
\newblock In {\em Set-Valued Analysis}, pages 1--27. Springer, 2008.
\newblock \url{https://doi.org/10.1007/978-0-8176-4848-0_10}.

\bibitem{Aumann1965}
R.~J. Aumann.
\newblock Integrals of set-valued functions.
\newblock {\em Journal of Mathematical Analysis and Applications}, 12(1):1--12,
  1965.
\newblock \url{https://doi.org/10.1016/0022-247X(65)90049-1}.

\bibitem{ben2009robust}
A.~Ben-Tal, L.~El~Ghaoui, and A.~Nemirovski.
\newblock {\em Robust Optimization}.
\newblock Princeton University Press, Princeton, NJ, 2009.
\newblock \url{https://doi.org/10.1515/9781400831050}.

\bibitem{benaim2005stochastic}
M.~Bena{\"\i}m, J.~Hofbauer, and S.~Sorin.
\newblock Stochastic approximations and differential inclusions.
\newblock {\em SIAM Journal on Control and Optimization}, 44(1):328--348, 2005.
\newblock \url{https://doi.org/10.1137/S0363012904439301}.

\bibitem{benaim2012perturbations}
M.~Bena{\"\i}m, J.~Hofbauer, and S.~Sorin.
\newblock Perturbations of set-valued dynamical systems, with applications to
  game theory.
\newblock {\em Dynamic Games and Applications}, 2(2):195--205, 2012.
\newblock \url{https://doi.org/10.1007/s13235-012-0040-0}.

\bibitem{Billingsley1999}
P.~Billingsley.
\newblock {\em Convergence of Probability Measures}.
\newblock Wiley, July 1999.
\newblock \url{https://doi.org/10.1002/9780470316962}.

\bibitem{bolte2009tame}
J.~Bolte, A.~Daniilidis, and A.~Lewis.
\newblock Tame functions are semismooth.
\newblock {\em Mathematical Programming}, 117(1):5--19, 2009.
\newblock \url{https://doi.org/10.1007/s10107-007-0166-9}.

\bibitem{bolte2023subgradient}
J.~Bolte, T.~Le, and E.~Pauwels.
\newblock Subgradient sampling for nonsmooth nonconvex minimization.
\newblock {\em SIAM Journal on Optimization}, 33(4):2542--2569, 2023.
\newblock \url{https://doi.org/10.1137/22M1479178}.

\bibitem{bolte2021conservative}
J.~Bolte and E.~Pauwels.
\newblock Conservative set-valued fields, automatic differentiation, stochastic
  gradient methods and deep learning.
\newblock {\em Mathematical Programming}, 188(1):19--51, jul 2021.
\newblock \url{https://doi.org/10.1007/s10107-020-01501-5}.

\bibitem{CastaingValadier1977}
C.~Castaing and M.~Valadier.
\newblock {\em Convex Analysis and Measurable Multifunctions}, volume 580 of
  {\em Lecture Notes in Mathematics}.
\newblock Springer, Berlin, 1977.
\newblock \url{https://doi.org/10.1007/BFb0087685}.

\bibitem{Chen2012Smoothing}
X.~Chen.
\newblock Smoothing methods for nonsmooth, nonconvex minimization.
\newblock {\em Mathematical Programming}, 134(1):71--99, jun 2012.
\newblock \url{https://doi.org/10.1007/s10107-012-0569-0}.

\bibitem{daniilidis2011continuity}
A.~Daniilidis and J.~C.~H. Pang.
\newblock Continuity and differentiability of set-valued maps revisited in the
  light of tame geometry.
\newblock {\em Journal of the London Mathematical Society}, 83(3):637--658,
  2011.
\newblock \url{https://doi.org/10.1112/jlms/jdq084}.

\bibitem{davis2022conservative}
D.~Davis and D.~Drusvyatskiy.
\newblock Conservative and semismooth derivatives are equivalent for
  semialgebraic maps.
\newblock {\em Set-Valued and Variational Analysis}, 30(2):453--463, 2022.
\newblock \url{https://doi.org/10.1007/s11228-021-00594-0}.

\bibitem{DavisDrusvyatskiy2022}
D.~Davis and D.~Drusvyatskiy.
\newblock Graphical convergence of subgradients in nonconvex optimization and
  learning.
\newblock {\em Mathematics of Operations Research}, 47(1):209--231, 2022.
\newblock \url{https://doi.org/10.1287/moor.2021.1126}.

\bibitem{davis2020stochastic}
D.~Davis, D.~Drusvyatskiy, S.~Kakade, and J.~D. Lee.
\newblock Stochastic subgradient method converges on tame functions.
\newblock {\em Foundations of Computational Mathematics}, 20(1):119--154, 2020.
\newblock \url{https://doi.org/10.1007/s10208-018-09409-5}.

\bibitem{Debreu1967}
G.~Debreu.
\newblock Integration of correspondences.
\newblock In {\em Proceedings of the Fifth Berkeley Symposium on Mathematical
  Statistics and Probability, Volume 2: Contributions to Probability Theory,
  Part 1}, pages 351--372, Berkeley, 1967. University of California Press.

\bibitem{dellacherie2011probabilities}
C.~Dellacherie and P.-A. Meyer.
\newblock {\em Probabilities and Potential}.
\newblock North-Holland, Amsterdam, 1978.

\bibitem{DrusvyatskiyXiao2023}
D.~Drusvyatskiy and L.~Xiao.
\newblock Stochastic optimization with decision-dependent distributions.
\newblock {\em Mathematics of Operations Research}, 48(2):954--998, 2023.
\newblock \url{https://doi.org/10.1287/moor.2022.1287}.

\bibitem{DuMordatch2019}
Y.~Du and I.~Mordatch.
\newblock Implicit generation and modeling with energy based models.
\newblock In H.~Wallach, H.~Larochelle, A.~Beygelzimer, F.~d'Alch{\'e} Buc,
  E.~Fox, and R.~Garnett, editors, {\em Advances in Neural Information
  Processing Systems}, volume~32. Curran Associates, Inc., 2019.

\bibitem{DuchiBartlettWainwright2012}
J.~C. Duchi, P.~L. Bartlett, and M.~J. Wainwright.
\newblock Randomized smoothing for stochastic optimization.
\newblock {\em SIAM Journal on Optimization}, 22(2):674--701, jan 2012.
\newblock \url{https://doi.org/10.1137/110831659}.

\bibitem{EnnajiFadiliAttouch2024}
H.~Ennaji, J.~M. Fadili, and H.~Attouch.
\newblock {Stochastic Monotone Inclusion with Closed Loop Distributions}.
\newblock {\em {Evolution Equations and Control Theory}}, 17:140--172, 2026.
\newblock \url{https://doi.org/10.3934/eect.2025022}.

\bibitem{ermol1998stochastic}
Y.~M. Ermol'ev and V.~I. Norkin.
\newblock Stochastic generalized gradient method for nonconvex nonsmooth
  stochastic optimization.
\newblock {\em Cybernetics and Systems Analysis}, 34(2):196--215, 1998.
\newblock \url{https://doi.org/10.1007/BF02742069}.

\bibitem{ErmolievNorkinWets1995}
Y.~M. Ermoliev, V.~I. Norkin, and R.~J.-B. Wets.
\newblock The minimization of semicontinuous functions: Mollifier subgradients.
\newblock {\em SIAM Journal on Control and Optimization}, 33(1):149--167, jan
  1995.
\newblock \url{https://doi.org/10.1137/S0363012992238369}.

\bibitem{Even2023}
M.~Even.
\newblock Stochastic gradient descent under markovian sampling schemes.
\newblock In A.~Krause, E.~Brunskill, K.~Cho, B.~Engelhardt, S.~Sabato, and
  J.~Scarlett, editors, {\em Proceedings of the 40th International Conference
  on Machine Learning}, volume 202 of {\em Proceedings of Machine Learning
  Research}, pages 9412--9439. PMLR, 2023.

\bibitem{HiaiUmegaki1977}
F.~Hiai and H.~Umegaki.
\newblock Integrals, conditional expectations, and martingales of multivalued
  functions.
\newblock {\em Journal of Multivariate Analysis}, 7(1):149--182, 1977.
\newblock \url{https://doi.org/10.1016/0047-259X(77)90037-X}.

\bibitem{LiuRomischXu2014}
Y.~Liu, W.~R{\"o}misch, and H.~Xu.
\newblock Quantitative stability analysis of stochastic generalized equations.
\newblock {\em SIAM Journal on Optimization}, 24(1):467--497, 2014.
\newblock \url{https://doi.org/10.1137/120880434}.

\bibitem{MeynTweedie2009}
S.~P. Meyn and R.~L. Tweedie.
\newblock {\em Markov Chains and Stochastic Stability}.
\newblock Cambridge University Press, Cambridge, 2 edition, 2009.
\newblock \url{https://doi.org/10.1017/CBO9780511626630}.

\bibitem{Molchanov2005}
I.~Molchanov.
\newblock {\em Theory of Random Sets}.
\newblock Probability and Its Applications. Springer, London, 2005.
\newblock \url{https://doi.org/10.1007/1-84628-150-4}.

\bibitem{NesterovSpokoiny2017}
Y.~Nesterov and V.~Spokoiny.
\newblock Random gradient-free minimization of convex functions.
\newblock {\em Foundations of Computational Mathematics}, 17(2):527--566, apr
  2017.
\newblock \url{https://doi.org/10.1007/s10208-015-9296-2}.

\bibitem{NorkinWets2013}
V.~I. Norkin and R.~J.-B. Wets.
\newblock On a strong graphical law of large numbers for random semicontinuous
  mappings.
\newblock {\em Vestnik Sankt-Peterburgskogo Universiteta, Seriya 10}, pages
  102--111, 2013.
\newblock No. 3.

\bibitem{PerdomoZrnicMendlerDunnerHardt2020}
J.~Perdomo, T.~Zrnic, C.~Mendler-D{\"u}nner, and M.~Hardt.
\newblock Performative prediction.
\newblock In H.~Daum{\'e}~III and A.~Singh, editors, {\em Proceedings of the
  37th International Conference on Machine Learning}, volume 119 of {\em
  Proceedings of Machine Learning Research}, pages 7599--7609. PMLR, 2020.

\bibitem{RamNedicVeeravalli2009}
S.~S. Ram, A.~Nedi{\'c}, and V.~V. Veeravalli.
\newblock Incremental stochastic subgradient algorithms for convex
  optimization.
\newblock {\em SIAM Journal on Optimization}, 20(2):691--717, 2009.
\newblock \url{https://doi.org/10.1137/080726380}.

\bibitem{RockafellarWets1998}
R.~T. Rockafellar and R.~J.-B. Wets.
\newblock {\em Variational Analysis}, volume 317 of {\em Grundlehren der
  mathematischen Wissenschaften}.
\newblock Springer, Berlin, 1998.
\newblock \url{https://doi.org/10.1007/978-3-642-02431-3}.

\bibitem{Ruan2025}
F.~Ruan.
\newblock On the uniform convergence of subdifferentials in stochastic
  optimization and learning.
\newblock {\em Mathematics of Operations Research}, 2025.
\newblock Articles in Advance. \url{https://doi.org/10.1287/moor.2024.0533}.

\bibitem{Salim2023}
A.~Salim.
\newblock A strong law of large numbers for random monotone operators.
\newblock {\em Set-Valued and Variational Analysis}, 31(4):38, nov 2023.
\newblock \url{https://doi.org/10.1007/s11228-023-00701-3}.

\bibitem{ShapiroXu2007}
A.~Shapiro and H.~Xu.
\newblock Uniform laws of large numbers for set-valued mappings and
  subdifferentials of random functions.
\newblock {\em Journal of Mathematical Analysis and Applications},
  325(2):1390--1399, 2007.
\newblock \url{https://doi.org/10.1016/j.jmaa.2006.02.078}.

\bibitem{TianRoyset2026}
L.~Tian and J.~O. Royset.
\newblock Failure of uniform laws of large numbers for subdifferentials and
  beyond, 2026.
\newblock \url{https://doi.org/10.48550/arXiv.2511.16568}.

\bibitem{XieLuZhuWu2016}
J.~Xie, Y.~Lu, S.-C. Zhu, and Y.~Wu.
\newblock A theory of generative {ConvNet}.
\newblock In M.~F. Balcan and K.~Q. Weinberger, editors, {\em Proceedings of
  the 33rd International Conference on Machine Learning}, volume~48 of {\em
  Proceedings of Machine Learning Research}, pages 2635--2644, New York, New
  York, USA, 20--22 Jun 2016. PMLR.

\bibitem{ZHANG1994355fubini}
D.~Zhang and C.~Guo.
\newblock Fubini theorem for {F}-valued integrals.
\newblock {\em Fuzzy Sets and Systems}, 62(3):355--358, 1994.
\newblock \url{https://doi.org/10.1016/0165-0114(94)90120-1}.

\end{thebibliography}

\end{document}